\documentclass[1p]{elsarticle}
\usepackage{amssymb, amsthm, amsmath}

\newtheorem{theorem}{Theorem}
\newtheorem{lemma}[theorem]{Lemma}

\newcommand{\im}{\mathop{\mathrm{im}}\nolimits}
\newcommand{\ex}{\mathop{\mathrm{ex}}\nolimits}

\begin{document}

\begin{frontmatter}

\title{On the maximum number of five-cycles in a triangle-free graph}
\author{Andrzej Grzesik}
\ead{Andrzej.Grzesik@uj.edu.pl}
\address{Theoretical Computer Science Department, Faculty of Mathematics and Computer Science, Jagiellonian University, ul. Prof. St. \L ojasiewicza 6, 30-348 Krak\'ow, Poland}

\begin{abstract}
Using Razborov's flag algebras we show that a triangle-free graph on $n$ vertices contains at most $\left(\frac{n}{5}\right)^5$ cycles of length five. It settles in the affirmative a conjecture of Erd\H{o}s.
\end{abstract}


\end{frontmatter}

In \cite{E}, Erd\H os conjectured that the number of cycles of  length 5 in a triangle-free graph of order $n$ is at most $(n/5)^5$
and further, this bound is attained in the case when $n$ is divisible by 5 by the blow-up of $C_5$ (i.e., five sets of $n/5$ independent vertices; vertices from different sets are connected according to the edges in $C_5$). Gy\H ori~\cite{G} showed that a triangle-free graph of order $n$
contains no more than $c\left(\frac{n+1}{5}\right)^5$ copies of $C_5$, where $c = \frac{16875}{16384} < 1.03$. Recently, this bound has been further improved by F\"uredi (personal communication). In this note, we settle the Erd\H{o}s conjecture in the affirmative using flag algebras.

Let us first introduce some basic notation, as well as recalling certain facts and results that we shall use later on. An $r$-graph $G$ is a pair $(V(G), E(G))$, where $V(G)$ is a set of vertices and $E(G)$ is a family of $r$-element subsets of $V(G)$ called edges.
For a (large) $r$-graph $G$ on $n$ vertices and a (small) $r$-graph $A$  on $k$ vertices let $C_A(G)$ be the set of all $k$-element subsets of $V(G)$ which induce a copy of $A$ in $G$.
The \textit{$A$-density} of an $r$-graph $G$ is defined as
$$d_A(G) = \frac{|C_A(G)|}{{n \choose k}}.$$
Thus, if $A$ is just a single edge, $d_A(G)$ becomes the standard edge density.

For a family $\mathcal{F}$ of forbidden $r$-graphs, we define the \textit{Tur\'an number} of $\mathcal{F}$ as
$$\ex_A(n,\mathcal{F}) = \max\{|C_A(G)| : G \mbox{ is }\mathcal{F}\mbox{-free, } |V(G)|=n\}$$
and by the \textit{Tur\'an density} $\pi_A(\mathcal{F})$ of $\mathcal{F}$ we mean
$$\pi_A(\mathcal{F}) = \lim_{n\rightarrow\infty}\frac{\ex_A(n,\mathcal{F})}{{n \choose k}}.$$

It is easy to show, using a `blow-up' argument similar to
 the one we use in the proof of Theorem~\ref{thm2} below,
that the limit above exists.

Let us now sketch the main idea behind the flag algebra approach
introduced by Razborov~\cite{R1} (see also Baber and Talbot~\cite{BT}).

Let us fix some $r$-graph $A$ on $k$ vertices and  let $\mathcal{F}$ be a
family of $r$-graphs whose Tur\'an density we wish to compute or
bound from above.

To this end, we consider the family $\mathcal{H}$ of all $\mathcal{F}$-free $r$-graphs on $l$ vertices, up to isomorphism. Clearly, if $l$ is small, we can list all elements $\mathcal{H}$, either by hand or by computer search.

For $H\in\mathcal{H}$ and a large $\mathcal{F}$-free $r$-graph $G$, let $p(H;G)$ denote the probability that a randomly chosen $l$-element set
from $V(G)$ induces a subgraph isomorphic to $H$. Thus, $d_A(G)=p(A;G)$. By averaging over all $l$-element subsets of $V(G)$,
we can express the $A$-density of $G$ as
\begin{equation}\label{eq1}
d_A(G) = \sum_{H\in\mathcal{H}}d_A(H)p(H;G),
\end{equation}
and hence,
\begin{equation}\label{eq2}
d_A(G) \leq \max_{H\in\mathcal{H}}d_A(H).
\end{equation}

The above bound on $d_A(G)$ is, in general, rather poor.
Using the flag algebra approach, one can sometimes improve on it significantly.

Thus, let us define a \textit{type} $\sigma = (G_\sigma, \theta)$ as an $r$-graph $G_\sigma$, together with a bijective map $\theta : [|\sigma|] \longrightarrow V(G_\sigma)$. By the \textit{order} $|\sigma|$ of $\sigma$ we mean $|V(G_\sigma)|$. Given a type $\sigma$, we define a \textit{$\sigma$-flag} $F=(G_F, \theta)$ as an $r$-graph with an injective  map $\theta$ that induces the type $\sigma$. A flag $F=(G_F, \theta)$ is called \textit{admissible} if $G_F$ is $\mathcal{F}$-free.

In other words, for a given family $\mathcal{F}$ and a type $\sigma$
(i.e., an $r$-graph with all vertices labelled by
$[|\sigma|]=\{1,\dots, |\sigma|\}$)
an admissible $\sigma$-flag of order $m$ is  an $\mathcal{F}$-free $r$-graph
on $m$ vertices, which has $|\sigma|$ labelled vertices which induce $\sigma$.

Let us fix a type $\sigma$ and an integer $m \leq (l+|\sigma|)/2$. This bound on $m$ ensures that an $r$-graph on $l$ vertices can contain two subgraphs on $m$ vertices overlapping in exactly $|\sigma|$ vertices. Let $\mathcal{F}_m^\sigma$ be the set of all admissible $\sigma$-flags of order $m$, up to isomorphism. Furthermore, by $\Theta_H$ we denote the set of all injections from $[|\sigma|]$ to $V(H)$. Finally, for  $F_a, F_b \in \mathcal{F}_m^\sigma$ and $\theta \in \Theta_H$, let $p(F_a, F_b, \theta_H; H)$  be the probability that if we choose a random $m$-element set $V_a \subseteq V(H)$ with $\im(\theta) \subset V_a$ and then select a random $m$-element set $V_b \subseteq V(H)$
such that $\im(\theta) = V_a \cap V_b$, then the induced $\sigma$-flags obtained are isomorphic to $F_a$ and $F_b$ respectively.

Consider a positive semidefinite square matrix $Q = (q_{ab})$ of  dimension $|\mathcal{F}_m^\sigma|$ and set
$$c_H(\sigma, m, Q) = \sum_{F_a,F_b\in\mathcal{F}_m^\sigma}q_{ab}\mathbb{E}_{\theta\in\Theta_H}[p(F_a, F_b, \theta; H)].$$

The following fact (see \cite{BT} or \cite{R2}) is crucial for
the flag algebra approach.

\begin{lemma}
For $t$ types $\sigma_i$, $m_i \leq (l+|\sigma_i|)/2$, positive semidefinite matrices $Q_i$ of dimension $|\mathcal{F}_{m_i}^{\sigma_i}|$,
and  $H\in\mathcal{H}$ let
$$c_H = \sum_{i=1}^t c_H(\sigma_i, m_i, Q_i).$$
Then
$$\sum_{H\in\mathcal{H}}c_H p(H;G) + o(1) \geq 0.$$
\end{lemma}

If we combine the above lemma with (\ref{eq1}) we get
$$d_A(G) \leq \sum_{H\in\mathcal{H}}(d_A(H) + c_H)p(H;G) + o(1).$$
Hence,
$$d_A(G) \leq \max_{H\in\mathcal{H}} (d_A(H) + c_H) + o(1)$$
and consequently
\begin{equation}\label{Turan bound}
\pi_A(\mathcal{F}) \leq \max_{H\in\mathcal{H}} (d_A(H) + c_H).
\end{equation}

Since $c_H$ may be negative, for appropriate choices of the $\sigma_i$,
$m_i$, and $Q_i$, this bound may be significantly better than the
bound given by (\ref{eq2}).

Note that now, one can bound the Tur\'an density by solving the following semidefinite programming problem:
given $\sigma_i$ and  $m_i$, we wish to find positive semidefinite matrices $Q_i$ which minimize the bound
on $\pi_A(\mathcal{F})$ given by (\ref{Turan bound}).

The main result of this note is given by the following theorem.

\begin{theorem}\label{thm1}
$\pi_{C_5}(K_3)\le \frac{24}{625}$.
\end{theorem}

\begin{proof}
Let us consider $l=5$ and three types on 3 vertices -- $\sigma_0$ stands for a graph with no edges, the type $\sigma_1$ has one edge and $\sigma_2$ has two. In each case, we consider $m=4$. There are 8 admissible $\sigma_0$-flags (the corresponding variables to these flags form the matrix $P$, say), 6 admissible $\sigma_1$-flags (we denote the corresponding matrix by $Q$) and 5 admissible $\sigma_2$-flags (matrix $R$). All of these graphs and flags are presented in Figure \ref{flagi}.
\begin{figure}[h]\label{flagi}
Triangle-free graphs on 5 vertices:\quad \begin{minipage}[c]{6ex}\centering\includegraphics[width=6ex]{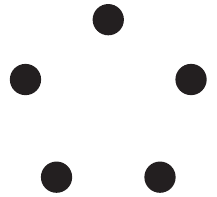}\end{minipage}\quad \begin{minipage}[c]{6ex}\centering\includegraphics[width=6ex]{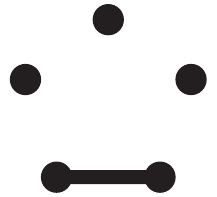}\end{minipage}\quad \begin{minipage}[c]{6ex}\centering\includegraphics[width=6ex]{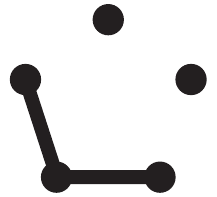}\end{minipage}\quad \begin{minipage}[c]{6ex}\centering\includegraphics[width=6ex]{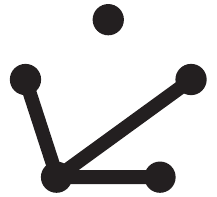}\end{minipage}\quad \begin{minipage}[c]{6ex}\centering\includegraphics[width=6ex]{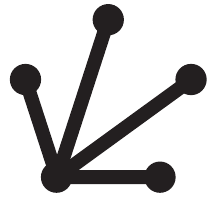}\end{minipage} \vspace{3mm} \\
\phantom{.}\qquad \quad \begin{minipage}[c]{6ex}\centering\includegraphics[width=6ex]{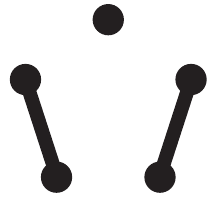}\end{minipage}\quad \begin{minipage}[c]{6ex}\centering\includegraphics[width=6ex]{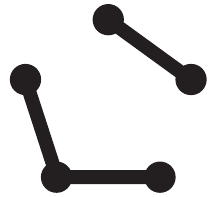}\end{minipage}\quad \begin{minipage}[c]{6ex}\centering\includegraphics[width=6ex]{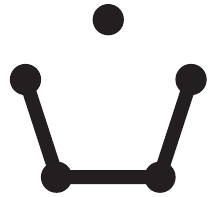}\end{minipage}\quad \begin{minipage}[c]{6ex}\centering\includegraphics[width=6ex]{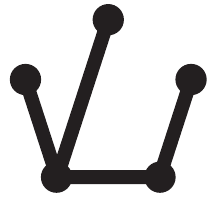}\end{minipage}\quad \begin{minipage}[c]{6ex}\centering\includegraphics[width=6ex]{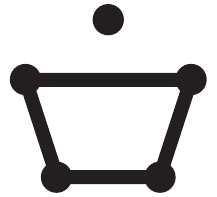}\end{minipage}\quad \begin{minipage}[c]{6ex}\centering\includegraphics[width=6ex]{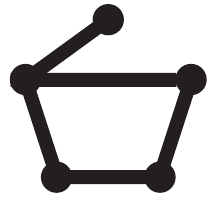}\end{minipage}\quad \begin{minipage}[c]{6ex}\centering\includegraphics[width=6ex]{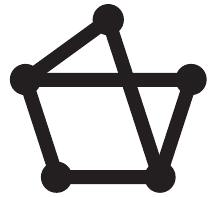}\end{minipage}\quad \begin{minipage}[c]{6ex}\centering\includegraphics[width=6ex]{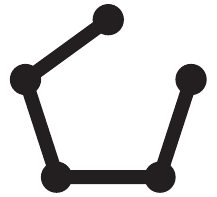}\end{minipage}\quad \begin{minipage}[c]{6ex}\centering\includegraphics[width=6ex]{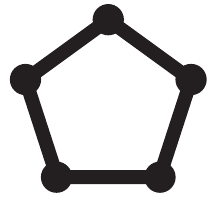}\end{minipage} \vspace{3mm} \\
$\sigma_0$-flags:\quad \begin{minipage}[c]{6ex}\centering\includegraphics[width=6ex]{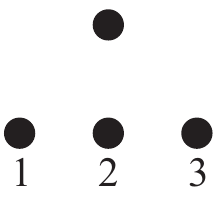}\end{minipage} \quad \begin{minipage}[c]{6ex}\centering\includegraphics[width=6ex]{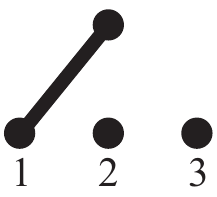}\end{minipage} \quad \begin{minipage}[c]{6ex}\centering\includegraphics[width=6ex]{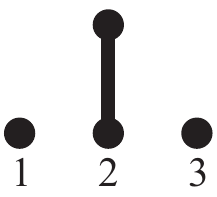}\end{minipage} \quad \begin{minipage}[c]{6ex}\centering\includegraphics[width=6ex]{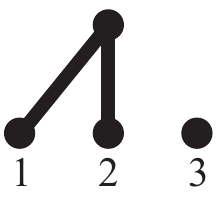}\end{minipage} \quad \begin{minipage}[c]{6ex}\centering\includegraphics[width=6ex]{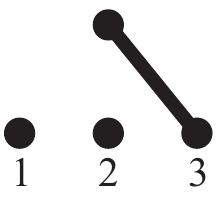}\end{minipage} \quad \begin{minipage}[c]{6ex}\centering\includegraphics[width=6ex]{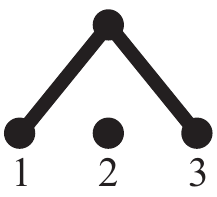}\end{minipage} \quad \begin{minipage}[c]{6ex}\centering\includegraphics[width=6ex]{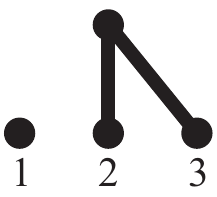}\end{minipage} \quad \begin{minipage}[c]{6ex}\centering\includegraphics[width=6ex]{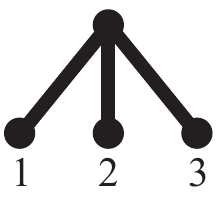}\end{minipage} \vspace{1mm} \\
$\sigma_1$-flags:\quad \begin{minipage}[c]{6ex}\centering\includegraphics[width=6ex]{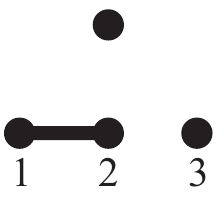}\end{minipage} \quad \begin{minipage}[c]{6ex}\centering\includegraphics[width=6ex]{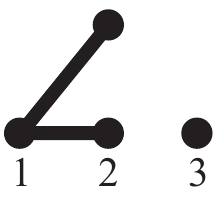}\end{minipage} \quad \begin{minipage}[c]{6ex}\centering\includegraphics[width=6ex]{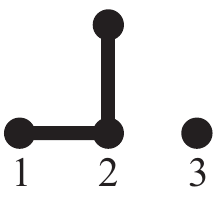}\end{minipage} \quad \begin{minipage}[c]{6ex}\centering\includegraphics[width=6ex]{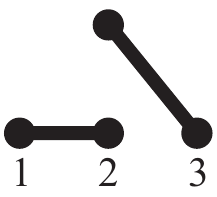}\end{minipage} \quad \begin{minipage}[c]{6ex}\centering\includegraphics[width=6ex]{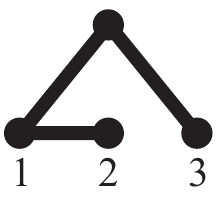}\end{minipage} \quad \begin{minipage}[c]{6ex}\centering\includegraphics[width=6ex]{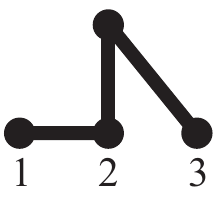}\end{minipage} \vspace{1mm} \\
$\sigma_2$-flags:\quad \begin{minipage}[c]{6ex}\centering\includegraphics[width=6ex]{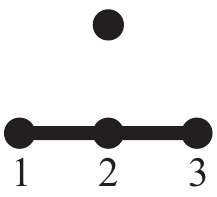}\end{minipage} \quad \begin{minipage}[c]{6ex}\centering\includegraphics[width=6ex]{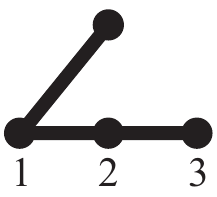}\end{minipage} \quad \begin{minipage}[c]{6ex}\centering\includegraphics[width=6ex]{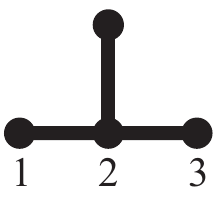}\end{minipage} \quad \begin{minipage}[c]{6ex}\centering\includegraphics[width=6ex]{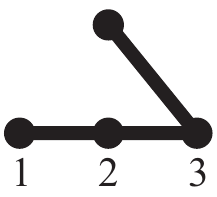}\end{minipage} \quad \begin{minipage}[c]{6ex}\centering\includegraphics[width=6ex]{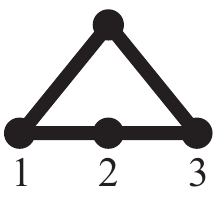}\end{minipage}
\caption{List of the considered graphs and flags.}\label{flagi}
\end{figure}

Computing all the appearances of each pair of flags in each graph, we infer that

\begin{align*}
\pi_{C_5}&(K_3) \leq \frac{1}{120}\max\{120p_{11},\\
    &12p_{11} + 24p_{12} + 24p_{13} + 24p_{15} + 12q_{11},\\
    &8p_{12} + 8p_{13} + 8p_{14} + 8p_{15} + 8p_{16} + 8p_{17} + 4p_{22} + 4p_{33} + 4p_{55} + 8q_{12} + 8q_{13} + 4r_{11},
\end{align*}
\begin{align*}
\phantom{\pi_{C_5}}
    &12p_{14} + 12p_{16} + 12p_{17} + 12p_{18} + 6q_{22} + 6q_{33} + 12r_{13},\\
    &48p_{18} + 24r_{33},\\
    &16p_{23} + 16p_{25} + 16p_{35} + 8q_{11} + 16q_{14},\\
    &8p_{27} + 8p_{36} + 8p_{45} + 8q_{14} + 8q_{24} + 8q_{34} + 4q_{44} + 4r_{11},\\
    &4p_{23} + 4p_{24} + 4p_{25} + 4p_{26} + 4p_{34} + 4p_{35} + 4p_{37} + 4p_{56} + 4p_{57} + 4q_{12} + 4q_{13} + 4q_{15} +\\
    &\qquad + 4q_{16} + 4q_{23} + 4r_{12} + 4r_{14},\\
    &4p_{27} + 4p_{28} + 4p_{36} + 4p_{38} + 4p_{45} + 4p_{58} + 4q_{15} + 4q_{16} + 4q_{25} + 4q_{36} + 4r_{13} + 2r_{22} +\\
    &\qquad + 4r_{23} + 4r_{34} + 2r_{44},\\
    &8p_{44} + 8p_{66} + 8p_{77} + 16q_{23} + 16r_{15},\\
    &4p_{48} + 4p_{68} + 4p_{78} + 4q_{26} + 4q_{35} + 2q_{55} + 2q_{66} + 4r_{15} + 4r_{23} + 4r_{25} + 4r_{34} + 4r_{35} +\\
    &\qquad + 4r_{45},\\
    &12p_{88} + 24r_{35} + 12r_{55},\\
    &4p_{46} + 4p_{47} + 4p_{67} + 4q_{24} + 4q_{26} + 4q_{34} + 4q_{35} + 4q_{45} + 4q_{46} + 4r_{12} + 4r_{14} + 4r_{24},\\
    &20q_{56} + 20r_{24} + 120\},
\end{align*}
where the maximum is taken over all possible coefficients $p_{ij}$, $q_{ij}$, $r_{ij}$ such that all the respective matrices $P$, $Q$, and $R$ are positive semidefinite.

For an explanation, we will consider one example and count appearances of each pair of flags in the second graph (consisting of a single edge). We need to consider all $5! = 120$ possibilities of placing three labels inducing a type and one additional label for each vertex left in the flags. If we place labels 1 and 2 in such a way that they are connected by an edge (there are 12 such possibilities), we always get two flags \begin{minipage}[c]{4ex}\centering\includegraphics[width=3ex]{flaga11.pdf}\end{minipage}, yielding coefficient $12q_{11}$. If labels 1 and 3 or 2 and 3 are connected by an edge, we get pairs of flags which are not under our consideration. If labels 1 and 4 or 1 and 5 are connected by an edge, we get flag \begin{minipage}[c]{4ex}\centering\includegraphics[width=3ex]{flaga02.pdf}\end{minipage} and flag \begin{minipage}[c]{4ex}\centering\includegraphics[width=3ex]{flaga01.pdf}\end{minipage}, yielding coefficient $24p_{12}$. Similarly, we get coefficients $24p_{13}$ and $24p_{15}$. The only possibilities left are those when 4 and 5 are connected by an edge. In those situations we get two flags \begin{minipage}[c]{4ex}\centering\includegraphics[width=3ex]{flaga01.pdf}\end{minipage}, yielding coefficient $12p_{11}$. Hence, we get $\frac{1}{120}(12p_{11} + 24p_{12} + 24p_{13} + 24p_{15} + 12q_{11})$. 

We take $P$, $Q$ and $R$ to be the matrices
$$P = \frac{1}{625}\left(
  \begin{array}{rrrrrrrr}
  24 & -36 & -36 & 24 & -36 & 24 & 24 & -36 \\
    -36 & 277 & 97 & -79 & 97 & -79 & -259 & 54 \\
    -36 & 97 & 277 & -79 & 97 & -259 & -79 & 54 \\
    24 & -79 & -79 & 247 & -259 & 67 & 67 & -36 \\
    -36 & 97 & 97 & -259 & 277 & -79 & -79 & 54 \\
    24 & -79 & -259 & 67 & -79 & 247 & 67 & -36 \\
    24 & -259 & -79 & 67 & -79 & 67 & 247 & -36 \\
    -36 & 54 & 54 & -36 & 54 & -36 & -36 & 54
  \end{array}
\right),$$
$$Q = \frac{1}{2500}\left(
  \begin{array}{rrrrrr}
    1728 & -1551 & -1551 & -1308 & 687 & 687 \\
    -1551 & 2336 & 742 & 908 & 2557 & -4084 \\
    -1551 & 742 & 2336 & 908 & -4084 & 2557 \\
    -1308 & 908 & 908 & 1728 & -254 & -254 \\
    687 & 2557 & -4084 & -254 & 15264 & -14424 \\
    687 & -4084 & 2557 & -254 & -14424 & 15264
  \end{array}
\right),$$
$$R = \frac{1}{625}\left(
  \begin{array}{rrrrr}
   1512 & 568 & -380 & 568 & -376 \\
   568 & 475 & -191 & 0 & -93 \\
   -380 & -191 & 192 & -191 & -2 \\
   568 & 0 & -191 & 475 & -93 \\
   -376 & -93 & -2 & -93 & 190
  \end{array}
\right).$$

It can be checked by any program for mathematical calculations (e.g. Mathematica, Maple) that matrix $P$ multiplied by 625 has characteristic polynomial
$$x^4(x-360)^2(x^2 - 930x + 53766),$$
and so it has eigenvalues $0$, $0$, $0$, $0$, $\approx62$, $360$, $360$ and $\approx 868$, matrix $Q$ multiplied by 2500 has characteristic polynomial $$x(x^2 - 31282x + 3219791)(x^3 - 7374x^2 + 7536419x - 324955440)$$
and eigenvalues $0$, $\approx 45$, $\approx 103$, $\approx 1170$, $\approx 6159$, $\approx 31179$, and matrix $R$ multiplied by 625 has characteristic polynomial
$$-x^2(x-475)(x^2 - 2369x + 492426)$$
and eigenvalues $0$, $0$, $\approx 230$, $475$ and $\approx 2139$. Thus, $P$, $Q$ and $R$ are all positive semidefinite.

Hence, for an upper bound on $\pi_{C_5}(K_3)$ we get
\begin{eqnarray*}
\pi_{C_5}(K_3) &\leq& \max\left\{\frac{24}{625}, \frac{24}{625}, \frac{24}{625}, \frac{24}{625}, \frac{24}{625}, \frac{24}{625}, \frac{322}{9375}, \frac{2355}{62500}, \frac{24}{625}, \frac{24}{625}, \right. \\
&& \left. \frac{24}{625}, \frac{24}{625}, -\frac{126}{6250}, \frac{24}{625}\right\} = \frac{24}{625}.
\end{eqnarray*}

\end{proof}

The Erd\H{o}s conjecture is a straightforward consequence of the above result.

\begin{theorem}\label{thm2}
The number of cycles of  length 5 in a triangle-free graph of order $n$ is at most $\left(\frac{n}{5}\right)^5$.
\end{theorem}

\begin{proof}
Suppose that there is a triangle-free graph $G$ on $n$ vertices which has at least $\left(\frac{n}{5}\right)^5+\varepsilon$ cycles $C_5$, where $\varepsilon>0$. Then, we can construct triangle-free graphs $G_{nN}$ consisting of $n$ sets of $N$ independent vertices and all edges between vertices in different sets according to the edges in $G$.

The graph $G_{nN}$ has $nN$ vertices and at least $\left(\left(\frac{n}{5}\right)^5+\varepsilon\right)N^5$ cycles $C_5$. Thus, the Tur\'an density is at least
$$\pi_{C_5}(K_3) \geq \lim_{N\rightarrow\infty}
\frac{\left(\frac{nN}{5}\right)^5+\varepsilon N^5}{{nN \choose 5}}
= \frac{24}{625} + \frac{120\varepsilon}{n^5}
> \frac{24}{625},$$
which contradicts Theorem~\ref{thm1}.
\end{proof}

\section*{Acknowledgement}

After writing this manuscript, I learnt that the main result was obtained, independently and simultaneously, by Hatami, Hladk\'y, Kr\'al, Norine and Razborov \cite{HHKNR} using different approach.

I would like to thank Tomasz \L uczak for  introducing me to the topic of flag algebras, as well as for his suggestion to use it to solve the Erd\H os pentagon conjecture, and for many valuable comments and remarks.


\begin{thebibliography}{9}
\bibitem{BT} R. Baber, J. Talbot, \textit{Hypergraphs do jump}, Combinatorics, Probability and Computing \textbf{20} (2011) 161-171.
\bibitem{E} P. Erd\H os, \textit{On some problems in graph theory, combinatorial analysis and combinatorial number theory}, Graph Theory and Combinatorics, Proc. Conf. Hon. P. Erd\H os, Cambridge 1983, 1-17 (1984).
\bibitem{G} E. Gy\H ori, \textit{On the number of $C_5$'s in a triangle-free graph}, Combinatorica \textbf{9} (1989) 101-102.
\bibitem{HHKNR} H. Hatami, J. Hladk\'y, D. Kr\'al, S. Norine, A. Razborov, \textit{On the Number of Pentagons in Triangle-Free
Graphs}, arXiv:1102.1634, 2011.
\bibitem{R1} A. Razborov, \textit{Flag Algebras}, Journal of Symbolic Logic \textbf{72} (2007) 1239-1282.
\bibitem{R2} A. Razborov, \textit{On 3-hypergraphs with forbidden 4-vertex configurations}, SIAM Journal on Discrete Mathematics \textbf{24} (2010) 946-963.
\end{thebibliography}
\end{document}